\title{Behavior of Fr\'{e}chet mean and Central Limit Theorems on spheres}
\author{Do Tran\thanks{
                  Department of Mathematics, Duke University, dotran@math.duke.edu}
        }
\documentclass{article}
\usepackage{amsmath}
\usepackage{amssymb}
\usepackage{hyperref}
\usepackage{graphicx}
\usepackage{mathrsfs}
\usepackage{color}
\usepackage{mathtools}
\usepackage{breqn}
\usepackage{amsthm}
\DeclarePairedDelimiter\norm{\lVert}{\rVert}%
\newtheorem{theorem}{Theorem}[section]

\newtheorem{proposition}[theorem]{Proposition}
\newtheorem{corollary}[theorem]{Corollary}

\numberwithin{equation}{section}

\theoremstyle{definition}
\newtheorem{definition}[theorem]{Definition}
\newtheorem{remark}[theorem]{Remark}
\newtheorem{expl}[theorem]{Example}
\newtheorem{assumption}[theorem]{Assumption}

\newenvironment{AMS}{}{}
\newenvironment{keywords}{}{}
\newcommand\CC{\mathcal{C}}
\newcommand\NN{\mathcal{N}}

\newcommand\XX{\mathcal{X}}

\newcommand\RR{\mathbb{R}}

\newcommand\BI{\boldsymbol{I}}

\newcommand\pp{\mathscr{P}}

\newcommand\dd{\textbf{d}}

\newcommand\muu{\bar\mu}

\newcommand\fc{\text{Fr\'{e}chet }}

\newcommand{\<}{\langle}
\renewcommand{\>}{\rangle}

\begin{document}
\newpage
\maketitle
%%%%%%%%%%%%%%%%%%%%%%%%%%%
% abstract, keywords and Subject classification are optional.
%%%%%%%%%%%%%%%%%%%%%%%%%%%
\begin{abstract}
We compute higher derivatives of the
Fr\'{e}chet function on spheres with an absolutely continuous
and rotationally symmetric probability distribution.  Consequences include
(i)~a practical condition to test if the mode of the symmetric
distribution is a local Fr\'{e}chet mean; (ii)~a Central Limit Theorem
on spheres with practical assumptions and an explicit limiting
distribution; and (iii)~an answer to the question of whether the
smeary effect can occur on spheres with absolutely continuous and rotationally symmetric
distributions: with the method presented here, it can in dimension at least~$4$.
\end{abstract}


\begin{AMS}{AMS 2000 Subject Classification:}
  \emph{ 60F05, %
   62H11}%
\end{AMS}

%%%%%%%%%%%%%%%%%%%%%%
% % Here is the start of the Text
%%%%%%%%%%%%%%%%%%%%%%
\section{Introduction}\label{s:intro}
%To be filled in later

\indent%
Methods of deriving Central Limit Theorems (CLT) for intrinsic
Fr\'echet means on manifolds in \cite{BP05, HH15, BL17, EH18} often
rely on the Taylor expansion of the Fr\'echet function at the mean. Thus,
two scenarios can happen depending on the invertibility of the Hessian
of the Fr\'echet function. The first case, in which the Hessian is
invertible and positive definite, results in a CLT that behaves
classically as shown in \cite{BP05, BL17}. The other one, when the Hessian
vanishes at the mean, results in a non classical behavior called
\emph{smeary CLT} as first shown in \cite{HH15, EH18}. In
both cases, an assumption about the differentiability the Fr\'echet
function in a neighborhood of the mean is required. Otherwise, the
reference measure would be required to give no mass to an open
neighborhood of the cut locus of the mean \cite{BL17, EGH19}. The reason for this
assumption is that the squared distance function from a point is not
differentiable at its cut locus, which means that the derivative
tensors would blow up near the cut locus.

This problem is addressed here by giving a practical condition to test
if the Fr\'echet function is differentiable at a point on a sphere of
arbitrary dimension (Proposition~\ref{p:diffcondprob}). Specifically,
the central observation is that in dimension at least $4$ the Fr\'echet function has derivative
of order $4$ at a point $p$ if the reference measure has a bounded
distribution function (with respect to the volume measure) in an open
neighborhood of the antipode of~$p$.

In the classical CLT case on manifolds, an assumption about the
invertibility of the Hessian of the Fr\'echet function at the mean is also required. 
The CLTs in \cite{BP05, BL17} and a more general form in \cite{EH18} need explicitly computed derivatives of the Fr\'echet function to be made precise. To date, a CLT with an explicit limit distribution for Fr\'echet means is only available on circles
\cite{HH15} and on spheres with specific models considered in \cite{EH18}. 

Here we fill this gap, for the case of absolutely continuous and rotationally symmetric distributions on spheres, by giving
an explicit computation of derivative tensors of the \fc function at the mean (Proposition~\ref{prop:derivative_tensors}). Two consequences include a practical condition
for $p$ to be a local Fr\'echet mean (Proposition \ref{prop:behavior}) and a
Central Limit Theorem with an explicit limit distribution with
Fr\'echet mean~at~$p$  (Theorem~\ref{thm:CLT_classic}).
 
The smeary effect occurs when the Hessian of the Fr\'echet
function vanishes and a tensor obtained by differentiating an even
number of times is positive definite at the mean.  In those
cases, the sample mean fluctuates asymptotically at a scale of
$n^{\alpha}$ for some $\alpha <1/2$. To date, examples of smeary
CLT on circles and spheres were observed in \cite{HH15, EH18, Elt19}. In
those examples, the reference measures are singular at the mean.  A
natural question arises: can smeariness happen for non-singular measures?

The computation of derivative tensors of the \fc function in Proposition~\ref{prop:derivative_tensors} yields an answer to this
question on spheres, when the distribution is rotationally symmetric. In
particular, in that setup when the distribution is absolutely
continuous and rotationally symmetric about its mean, the smeary effect can only appear in dimension~$4$ or
higher and is more likely as the dimension grows
(Remark~\ref{r:smeary}). This is because of the fact that in low
dimensions, the volume element cannot outweigh the negativity of the
derivative tensors near the cut locus of the mean. As a result, the
derivative tensor of order $4$ is always negative in dimension~$2$
and~$3$. In higher dimensions, we give a sufficient condition for a
smeary CLT to occur with $\alpha = 1/6$ (Theorem~\ref{thm:smearyCLT}).

The first example of a smeary CLT on spheres with no mass in a neighborhood of the cut locus of the mean was shown in \cite{Elt19}. In that example, the reference measure gives a singular mass to the mean. Example~\ref{e:smeary} exhibits the smeary effect of a local Fr\'echet mean when the measure has no singular mass. Specifically, the measure has a uniformly distributed portion on a bottom cup and another uniformly distributed portion on a thin upper strip near the equator and zero everywhere else. The South pole is a local mean and the strip near the equator gets thinner as the dimension grows.

In contrast to the method of expressing the Frech\'et function
explicitly in a polar coordinate system as shown in \cite{HH15, EH18, Elt19}, the main method used in
this paper is applying Jacobi fields with a non-parametric approach to compute derivative tensors of
the squared distance function on spheres (Section
\ref{s:derv_tensors}). The rotationally symmetric condition of the
distribution then allows us to derive derivative tensors of the
Frech\'et function (Section \ref{s:differentiability}). 

It is also worth mentioning that the method in this paper does not
work well for models in spheres of dimension $2$ and $3$ in which
there is a singular mass at the cut locus of the mean. In such scenarios, writing the Frech\'et function
explicitly in a coordinate system yields better results, as shown in
\cite{EH18}.

\section{Basic notations and principal results}
\subsection{Basic notations}

Suppose that $M$ is a Riemannian manifold with geodesic length $\dd$ and $\mu$ is a probability measure on the Borel $\sigma$-algebra of $M$. For any point $y$ in $M$, set 
	\begin{equation} \label{eq:squared_distance}
		\rho_y(x)=\frac{1}{2}\dd^2(y,x)=\frac{1}{2}\dd_y^2(x).
	\end{equation} 
The \emph{\fc function} of $\mu$ is defined as 
	\[ F(x) =\int_M \rho_y(x) d\mu (y) . \]
We shall always assume that $F$ is finite on $M$. The \emph{(intrinsic)} \emph{Fr\'echet mean set} $\muu$ of $\mu$ is defined as 
	\begin{equation*}
		\muu=\{y \in M: F(y) \leq F(x), \text{for all } x \in M \},
	\end{equation*}
which can be summarized as $\muu =\arg \min_{x}F(x)$. In general, we use $\bar \nu$ to denote the mean set of a probability measure $\nu$.

Consider i.i.d. $M$-valued random variables $X_1,\ldots, X_n$ with the common distribution $\mu$. The \emph{sample \fc function}, which is denoted by $F_n$, is the \fc function of the \emph{empirical distribution} $\mu_n=\frac{1}{2}\sum_{i=1}^n \delta_{X_i}$. The Fr\'echet mean set $\muu_n$ is called the \emph{(intrinsic) sample \fc mean set}. 
 
 We will use $\exp_p$ and $\log_p$ for the exponential map and the logarithm map at a point $p \in M$ and $\CC_{U}$ for the cut locus of a closed subset $U$ of $M$.
 
The following assumptions are familiar in developing CLTs for \fc means in \cite{BP05, BB08, BL17} and its smeary generalization in \cite{EH18}.
	\begin{enumerate}
		\item[A1.]	The \fc mean $\muu$ is unique. \label{cond:A1}
		\item[A2.]  	$\mu$ assigns zero mass to the cut locus of $\muu$, i.e. $\mu(\CC_{\muu})=0.$ \label{cond:A2}
		\item[A3.] $F$ is differentiable up to order $r$ in a neighborhood $U$ of $\muu$ and $\nabla^rF$ is the first non-zero derivative tensor at $\muu$ of $F$. \label{cond:A3}
		\item[A4.] $\nabla ^rF_{\muu}$ is non-degenerate, i.e. there is no non-zero vector $Z \in T_{\muu}M$ such that $\nabla^r F_{\muu}(Z ^{\otimes r})=0$.\label{cond:A4}
	\end{enumerate}

The CLT for the intrinsic sample \fc mean $\muu_n$ in \cite{BP05, BL17, BB08} and its smeary generalization in \cite{EH18} can be summerized as follows.

\begin{theorem}[CLT for intrinsic mean]\label{thm:CLT_for_intrinsic_means} Under assumptions \hyperref[cond:A1]{A1}-\hyperref[cond:A4]{A4}, for any measurable choice of $\muu_n$
	\begin{equation}\label{eq:new_form_CLT}
		\lim_{n\to \infty} n^{1/(2k+2)} \log_{\muu}\muu_n \sim H_{\sharp}(\lim_{n \to \infty} n^{1/2} \overline{\log_{\muu, \sharp} \mu_n}),
	\end{equation}
where $k$ is the \emph{degree of smeariness}, $H: T_{\muu}M \to T_{\muu}M$ is a specific map which we call the \emph{correction map}, and $H_{\sharp}$ is the pushforward of measures induced by $H$. 
\end{theorem}
\begin{remark}
\begin{enumerate}
		
	\item In the classical CLT on Euclidean spaces and the CLT on manifolds in \cite{BP05}, the degree of smeariness $k$ is $0$, in other words, there is no smeary effect in those cases. In general, \cite[Theorem 2.11]{EH18} shows that $k=r-2$ with $r$ is the degree of the first non-zero derivative tensor at $\muu$ of $F$ in assumptions \hyperref[cond:A3]{A3} and \hyperref[cond:A4]{A4}.

	\item In the CLT in \cite{BP05}, the correction map $H$ is the inverse of the Hessian of the \fc function at $\muu$, i.e. $H=(\nabla^2F_{\muu})^{-1}$.

	\item In the smeary generalization CLT in \cite{EH18}, the correction map $H$ is the inverse of the following map
		\begin{equation} \label{eq:tau_map}
		\begin{split}
			\tau : T_{\muu}M &\to T_{\muu}M\\
			Z &\mapsto  \frac{1}{(r-1)!}\nabla^{r}F_{\muu}(Z^{\otimes (r-1)}).
		\end{split}	
		\end{equation}
In order to give an explicit form of $H$, the following assumption was used in \cite{EH18} instead of assumption \hyperref[cond:A4]{A4}.

\begin{enumerate}
		\item[A5.]  There is an orthonormal basis $\{ e_1,\ldots, e_d \}$ of $T_{\muu}M$ and $\lambda_i>0$ for all $i=1,\ldots, d$ such that
		\begin{equation*} \label{cond:A5}
			\nabla^rF_{\muu}=\sum_{i=1}^d r!\lambda_i e_i^{\otimes r}.	
		\end{equation*}
		In other words, $\nabla^r F$ is \emph{othorgonal decomposable}; see \cite{Rob16}.
	\end{enumerate}
With assumption \hyperref[cond:A5]{A5} in effect, the map $\tau$ in \ref{eq:tau_map} is 
	\begin{equation*}
		\begin{split}
			\tau: T_{\muu}M &\to T_{\muu}M \\
				\sum_{i=1}^d a_i e_i &\mapsto \sum_{i=1}^d r\lambda_i a_i^{r-1}e_i.
		\end{split}
	\end{equation*}
The correction map is then
	\begin{equation}
		\begin{split}
			H: T_{\muu}M &\to T_{\muu}M \\
			\sum_{i=1}^d a_i e_i &\mapsto \sum\limits_{i=1}^d\Big(\frac{a_i}{r\lambda_i}\Big)^{1/(r-1)}e_i.	
		\end{split}
	\end{equation}
Note here that assumption~\hyperref[cond:A5]{A5} is quite strong as shown in \cite{Rob16}. In addition, finding a decomposition of $\nabla^rF_{\muu}$ that satisfies assumption~\hyperref[cond:A5]{A5} is NP-hard \cite{HL13}. We will show in Remark~\ref{re:not_odoco} that assumption~\hyperref[cond:A5]{A5} does not hold when $M$ is the unit sphere and $\mu$ is absolutely continuous and rotationally symmetric about its mean $\muu$. 
\item Recently, it has been shown in \cite{my_thesis} that the correction map $H$ can be defined independently from the derivative tensor of $F$ as follows.
	For any positive $\epsilon$ and any point $x \in M$, write $\muu_{\epsilon, x}$ for the \fc mean of $(1-\epsilon)\mu+\epsilon \delta_x$. Define
		\begin{equation} \label{e:chap3_inverse_hess}
			\begin{split}
				h:M	&\to T_{\muu}M\\
				x &\mapsto \lim \limits_{\epsilon \to 0} \frac{\log_{\muu}\muu_{\epsilon,x}}{\epsilon^{1/(r-1)}}.
			\end{split}
		\end{equation}
	The correction map $H$ is the composition
		\begin{equation*}\label{e:chap3_inverse_hess_2}
			\begin{split}
				H: T_{\muu}M &\to T_{\muu}M\\
				X &\mapsto h(\exp_{\muu}X). 
			\end{split}
		\end{equation*}

\end{enumerate}
\end{remark}

The left side of Eq.~\eqref{eq:new_form_CLT} is the asymptotic distribution of a rescaled of the image of the intrinsic sample mean $\muu_n$ under the logarithm map $\log_{\muu}$. In other words, the left side of Eq~\ref{eq:new_form_CLT} is the asymptotic distribution of the intrinsic sample mean. 

On the right side of Eq.~\eqref{eq:new_form_CLT}, $\log_{\muu,\sharp} \mu_n$ is the pushforward of the empirical measure $\mu_n$ under the logarithm map $\log_{\muu}$ and $\overline{\log_{\muu, \sharp} \mu_n}$ is the \fc mean of $\log_{\muu,\sharp} \mu_n$. We call $\overline{\log_{\muu, \sharp} \mu_n}$ the \emph{tangent sample mean}. The classical CLT for $\log_{\muu,\sharp}\mu$ on the Euclidean space $T_{\muu}M$ implies that $n^{1/2} \overline{\log_{\muu, \sharp} \mu_n}$ is asymptotic normal with distribution $\NN$. 

\begin{definition}\label{defi:tangental_CLT}
	We call the normal law
		\[\NN \sim  \lim \limits_{n \to \infty } n^{1/2} \overline{\log_{\muu, \sharp} \mu_n} \]
	the \emph{tangent limit distribution} of $\mu$.
\end{definition}

In words, the CLT for intrinsic mean (Theorem~\ref{thm:CLT_for_intrinsic_means}) states that the asymptotic distribution of the intrinsic sample mean deviates from the tangent limit distribution and the deviation is encoded in correction map $H$. 

\subsection{Main results}

\subsubsection{Differentibility of \fc function}
Consider $M$ the unit sphere $S^d$ of dimension $d$ with the standard great circle geodesic distance $\dd$ and the measure $\mu$ is absolutely continuous with respect to the volume measure of $M$. The following proposition gives a practical condition to test assumption~\hyperref[cond:A3]{A3}---that is to test if the Fr\'echet function is differentiable at a point $p$.

\begin{proposition}\label{p:diffcondprob}%
Let $f$ be the distribution function of $\mu$. Fix a point $p \in S^d$, suppose that there exists an $\epsilon$-ball $B_{\epsilon}(\CC_p)$ centered at $\CC_p$ such that for $j \leq 4$
\begin{itemize}
	\item[i.] $f$ is bounded on $B_{\epsilon}(\CC_p)$ when $d \geq j$, or
	\label{cond:i}
	\item[ii.] $f$ vanishes on $B_{\epsilon}(\CC_p)$ when $d < j$. \label{cond:ii}
\end{itemize}
Then the \fc function $F$ is differentiable of order $j$ in an open neighborhood of $p$.
\end{proposition}

To compute the derivative tensors of the \fc function as in condition \hyperref[cond:A5]{A5}, we impose the following rotationally symmetric condition on $\mu$.

\begin{assumption}\label{assumption:1}
	The distribution function $f(\dd(p, \cdot ))$ of $\mu$ depends only on the distance $\dd(p, \cdot)$ from $p$, i.e.
 		\[ d\mu(q)=f(\dd(p,q))dV(q), \]
 	where $dV$ is the volume element.
\end{assumption}
\begin{remark}
	The point p in the above assumption will often be the \fc mean $\muu$. In such cases, we say $\mu$ is \emph{absolutely continuous and rotationally symmetric about its mean}.
\end{remark}
\noindent Under Assumption~\ref{assumption:1}, set 
	\begin{equation}
			\begin{split}
				\alpha_d &=\frac{V(S^{d-1})}{d}\int_0^{\pi}f(\varphi)d(\varphi \sin^{d-1}\varphi),\\
				\beta_d &=\frac{a_dV(S^{d-2})}{d+2}\int_0^{\pi}f(\varphi)d\Big(\sin^{d-3}\varphi \big(2\varphi\sin^2\varphi+3\cos \varphi \sin \varphi-3\varphi \big) \Big). 
			\end{split}	
		\end{equation}

The following result enables us to give explicit formulas of the correction map $H$ and hence of the CLT for Fr\'echet means on spheres.

\begin{proposition} \label{prop:derivative_tensors}
	Suppose that $\mu$ satisfies Assumption~\ref{assumption:1} and that that the Fr\'echet function $F$ is differentiable up to order $4$ in an open neighborhood of $p$. Let $Z$ be a vector in $T_{\muu}S^d$, then the derivative tensors of $F$ at $p$ are 
			\begin{equation*}
				\begin{split}
					\nabla F(Z)_p&=0; \ \nabla^2F(Z)_p=Z \alpha_d;\ \nabla^3F(Z,Z)=0;\ and \\
				 	\nabla^4F(Z,Z,Z)_p&=\begin{cases}Z \norm{Z}^2 \beta_d \text{ if } d\geq 4,\\
				 Z \norm{Z}^2 C \text{ with }C<0 \text{ if } d< 4.
				 				\end{cases}
				 \end{split}	
			\end{equation*}
\end{proposition}

\begin{remark} \label{re:not_odoco}
	Note here that the tensor $\nabla ^4F_p$ in Proposition~\ref{prop:derivative_tensors} is not orthogonal decomposable in the sense of assumption \hyperref[cond:A5]{A5}. Indeed, Proposition~\ref{prop:derivative_tensors} shows that $\nabla^4F(Z^{\otimes 4})_p=c\norm{Z(p)}^4 $ for some constant $c$. Thus, there is no orthonormal basis $\{e_1,\ldots, e_d \}$ of $T_{p}M$ such that for $Z(p)=\sum_{i=1}^da_i e_i$,
		\begin{equation*}
			\nabla^4F(Z^{\otimes 4})_p=\sum_{i=1}^d	\lambda_i a_i^4.
		\end{equation*}
\end{remark}

\subsubsection{Central Limit Theorems and the correction map}

The following CLT, which is a direct consequence of Proposition~\ref{prop:derivative_tensors} and \cite[Theorem 2.11]{EH18}, gives us an explicit formula of the correction map $H$ when there is no smeary effect.

\begin{theorem}\label{thm:CLT_classic}
	Suppose that $\muu$ is unique and $\mu$ is absolutely continuous and rotationally symmetric about $\muu$ (Assumption~\ref{assumption:1}). In addition, assume that 
	\begin{itemize} 
		\item[i.]  the function $f$ is bounded on an open neighborhood of the cut locus $\CC_p$; and
		\item[iii.] $\alpha_{d}>0$. 
	\end{itemize} 
With the tangent limit distribution $\NN$ defined in Definition~\ref{defi:tangental_CLT}, the CLT on $S^d$ has the following form
	  		\begin{equation}\label{eq:CLT_classic}
	  				 n^{1/2}\log_{\muu}\muu_n \xrightarrow {\mathcal{D}} H_{\sharp} \NN,
	  		\end{equation}
where the correction map $H$ is defined as
	\begin{equation} \label{eq:correction_inverse_hessian}
		\begin{split}
 			H:T_{\muu}M &\to T_{\muu}M \\
 				Z &\mapsto Z\alpha_d^{-1}.
 		\end{split}
 	\end{equation}
\end{theorem}

\begin{proof}
	Thanks to \cite[Theorem 2.11]{EH18}, it suffices to write the Taylor expansion of the Fr\'echet function at $\muu$. It follows from Proposition~\ref{prop:derivative_tensors} that for $q$ close to $\muu$,
	\[ F(q)=F(\muu)+\alpha_d \dd^2(\muu,q)+ o(\dd^2(\muu,q)). \]
	The CLT \eqref{eq:CLT_classic} now follows directly from \cite[Theorem 2.11]{EH18}.
\end{proof}

The second consequence of Proposition~\ref{prop:derivative_tensors} is that under some certain restriction, the smeary effect does not occur on spheres of dimension $2$ and $3$. 

\begin{proposition} \label{prop:behavior}
 Suppose that $\mu$ satisfies Assumption~\ref{assumption:1}, then the following claims hold true.
	 \begin{enumerate}
	 	\item If $f$ is bounded on $B_{\epsilon}(\CC_p)$ for some $\epsilon>0$ and $\alpha_d>0$ then $p$ is a local minimum of the Fr\'echet function $F$.
	 	\item If $d<4$ and $f$ vanishes on $B_{\epsilon}(\CC_p)$ for some $\epsilon>0$ and $\alpha_d=0$ then $p$ is a local maximum of the Fr\'echet function.
	 \end{enumerate}
\end{proposition}\label{prop:behavior}
\begin{proof}
	First note that the rotational symmetry of $\mu$ implies that $\nabla F_p =0$ and hence $p$ is a critical point of $F$.
	
	Suppose that the conditions in claim 1 hold. The boundedness of $f$ on $B_{\epsilon}(\CC_p)$ ensures that the Fr\'echet function has derivative up to order $2$ (Proposition~\ref{p:diffcondprob}). It then follows from Eq.~\eqref{2nddervre} that $\alpha_d>0$ implies $\nabla^2F(Z,Z)_p>0$ for all $Z \in \XX(S^d)$. Thus, $p$ is a local minimum of $F$.
	
	Suppose that the conditions in the second claim hold. Proposition~\ref{p:diffcondprob} tells us that the Fr\'echet function has derivative up to order $4$. It follows from Eq.~\eqref{3rddervfin}, Eq.~\eqref{fourthdervre}, and $\alpha_d=0$ that the restriction of the function $F$ to the geodesic tangent to $Z(p)$ has a local maximum at $p$. Since $Z$ can be chosen freely and $F$ is rotationally symmetric, $p$ is a local maximum of $F$.
\end{proof}

\begin{corollary} \label{cor:smeary}
	In dimensions $d=2$ or $3$, suppose that $\muu$ is unique and $\mu$ is absolutely continuous and rotationally symmetric about $\muu$ (Assumption~\ref{assumption:1}). In addition, $\mu$ gives zero mass to an open neighborhood of the cut locus of $\muu$. Then there is no smeary effect on the unit sphere $S^d$. 
\end{corollary}

When $d \geq 4$, there are scenarios in which $\alpha_d$ vanishes and $\beta_d$ is positive---that is, when the CLT \eqref{eq:new_form_CLT} has degree of smeariness $k=2$. From Proposition~\ref{prop:derivative_tensors}, the map $\tau$ in Eq.~\eqref{eq:tau_map} is defined as 
	\begin{equation}
		\begin{split}
			\tau:T_{\muu}M &\to T_{\muu}M\\
				Z &\mapsto  \frac{\beta_d}{6} Z \norm{Z}^2.
		\end{split}	
	\end{equation}
The correction map $H$ is the inverse of $\tau$, which is
		\begin{equation}\label{eq:correction_smeary}
		\begin{split} 
			H: T_{\muu}S^d &\to T_{\muu}S^d\\
					Z &\mapsto  Z\norm{Z}^{-2/3}\beta_d^{-1/3} 6^{1/3}.
		\end{split}
		\end{equation}

\begin{theorem} \label{thm:smearyCLT}
	Given $d\geq 4$, suppose that $\muu$ is unique and $\mu$ is absolutely continuous and rotationally symmetric about $\muu$ (Assumption~\ref{assumption:1}). In addition, assume the following conditions
	\begin{itemize}
		\item[i.] $f$ is bounded in an open neighborhood of the cut locus $\CC_{\muu}$; \label{cond:smeary_1}
		\item[ii.] $\alpha_{d}=0$; and \label{cond:smeary_2}
		\item[iii.] $\beta_d >0.$ \label{cond:smeary_3}
	\end{itemize} 
With the correction map defined in Eq.~\eqref{eq:correction_smeary} and the tangental limit distribution $\NN$ defined in Definition~\ref{defi:tangental_CLT}, we have the following $2$-smeary CLT
	\[ n^{1/6}\log_{\muu}\muu_n \xrightarrow {\mathcal{D}} H_{\sharp} \mathcal{N}.\]
\end{theorem}
\begin{proof}
Set $Z_n=\log_{\muu}\muu_n$. Then Proposition~\ref{prop:derivative_tensors} implies 
\begin{align} \label{6.2.2}
	\begin{split}
		\nabla^3F(Z_n,Z_n)&=0,\\
		\nabla^4F(Z_n,Z_n,Z_n)&=Z_n\norm{Z_n}^2\beta_d,
	\end{split}	
\end{align}
The Taylor expansion of $F$ at $\muu$ along $Z_n$ gives
\begin{equation} \label{eq:taylor1}
	\begin{split}
		F(\muu_n)&=F(\muu)+\nabla F(Z_n)+\frac{1}{2}\nabla^2F(Z_n,Z_n)+\frac{1}{6}\nabla^3F(Z_n,Z_n,Z_n)\\
		&\hspace{1em}+ \frac{1}{24}\nabla^4F(Z_n,Z_n,Z_n,Z_n) + o(\norm{Z_n}^4),
	\end{split}
\end{equation}
which is simplified to
\begin{align}\label{4th_der_symm}
	\begin{split}
		F(\muu_n)=F(\muu)+\frac{1}{24}\norm{Z_n}^4\beta_n+o(\norm{Z_n}^4). 
	\end{split}	
\end{align}
So
	\begin{equation*}
		\nabla F (\muu_n)=\nabla F (\muu)+\frac{1}{6}\nabla^4F(Z_n,Z_n,Z_n)+o(\norm{Z_n}^3),
	\end{equation*}
which, in conjunction with Proposition~\ref{prop:derivative_tensors}, yields
	\begin{equation} \label{eq:taylor_of_gradient}
		\nabla F (\muu_n)=\frac{1}{6}Z_n\norm{Z_n}^2\beta_d+o(\norm{Z_n}^3).
	\end{equation}
Observe that when $\beta_d \neq 0$, the map
		\begin{align*}
			\tau: T_{\muu}S^d &\to T_{\muu}S^d\\
					Z &\mapsto \frac{1}{6}Z\norm{Z}^2\beta_d
		\end{align*}
	has an inverse
		\begin{equation}\label{e:chap3_inverse_4th_derv}
		\begin{split} 
			H: T_{\muu}S^d &\to T_{\muu}S^d\\
					Z &\mapsto  Z\norm{Z}^{-2/3}\beta_d^{-1/3} 6^{1/3}.
		\end{split}
		\end{equation}

Now we apply Eqs.~\eqref{eq:taylor_of_gradient} and \eqref{e:chap3_inverse_4th_derv} to Theorem~\ref{thm:CLT_for_intrinsic_means} with $k=2$ and the correction map defined in Eq.~\eqref{e:chap3_inverse_4th_derv} to obtain a $2$-smeary CLT
 		\begin{equation*}
 			 n^{1/6}\log_{\muu}\muu_n \xrightarrow {\mathcal{D}} H_{\sharp}\mathcal{N}. 
 		\end{equation*}
 \vglue -2em
 \end{proof}
 
 \begin{remark}
At least for the case of constant sectional curvature, we showed from Corollary~\ref{cor:smeary} and Theorem~\ref{thm:smearyCLT} that the occurrence of the smeary effect depends not only on curvature but also on the dimension of the manifold, as long as its sectional curvature is positive.
\end{remark}
Conditions $\alpha_d=0$ and $\beta_d>0$ are not hard to achieve,
especially in high dimensions. Example~\ref{e:smeary} below exibits a piecewise constant function $f$ such that conditions \hyperref[cond:smeary_1]{(\emph{i})}, \hyperref[cond:smeary_2]{(\emph{ii})}, and \hyperref[cond:smeary_3]{(\emph{iii})} in Theorem~\ref{thm:smearyCLT} hold. However, in the example, we have to make use of assumption A1 since we are unable to verify the uniqueness of the mean. Examples of global Fr\'echet means with rotationally symmetric distribution and a singular mass at the mean are presented in \cite{Elt19}.

\begin{expl}\label{e:smeary}
We need the following conditions on $f(\varphi),$
\[ \int_0^{\pi}f(\varphi)d(\varphi \sin^{d-1}\varphi)=0, \text{ and }\]
\[ \int_0^{\pi}f(\varphi)d\Big(\sin^{n-3}\varphi \big(2\varphi\sin^2\varphi+3\cos \varphi \sin \varphi-3\varphi \big) \Big)>0.\]
These are equivalent to
\begin{equation} \label{expl5.6.1}
		\int_0^{\pi}f(\varphi)d(\varphi \sin^{d-1}\varphi)=0,
\end{equation}
\begin{equation} \label{expl5.6.2}
	\int_0^{\pi}f(\varphi)d\Big(\sin^{d-3}\varphi \cos \varphi \big(\sin \varphi-\varphi \cos \varphi \big) \Big)>0.
\end{equation}
Set $g_1(\varphi)=\varphi \sin^{d-1}\varphi$ and $g_2(\varphi)=\sin^{d-3}\varphi \cos \varphi \big(\sin \varphi-\varphi \cos \varphi \big)$. Note that $g_2$ is positive on $(0,\pi/2)$. The graphs of $g_2$ with $d=10$ are depicted below.
\begin{figure*}[h!]
\centering
	\includegraphics[width=0.65\textwidth]{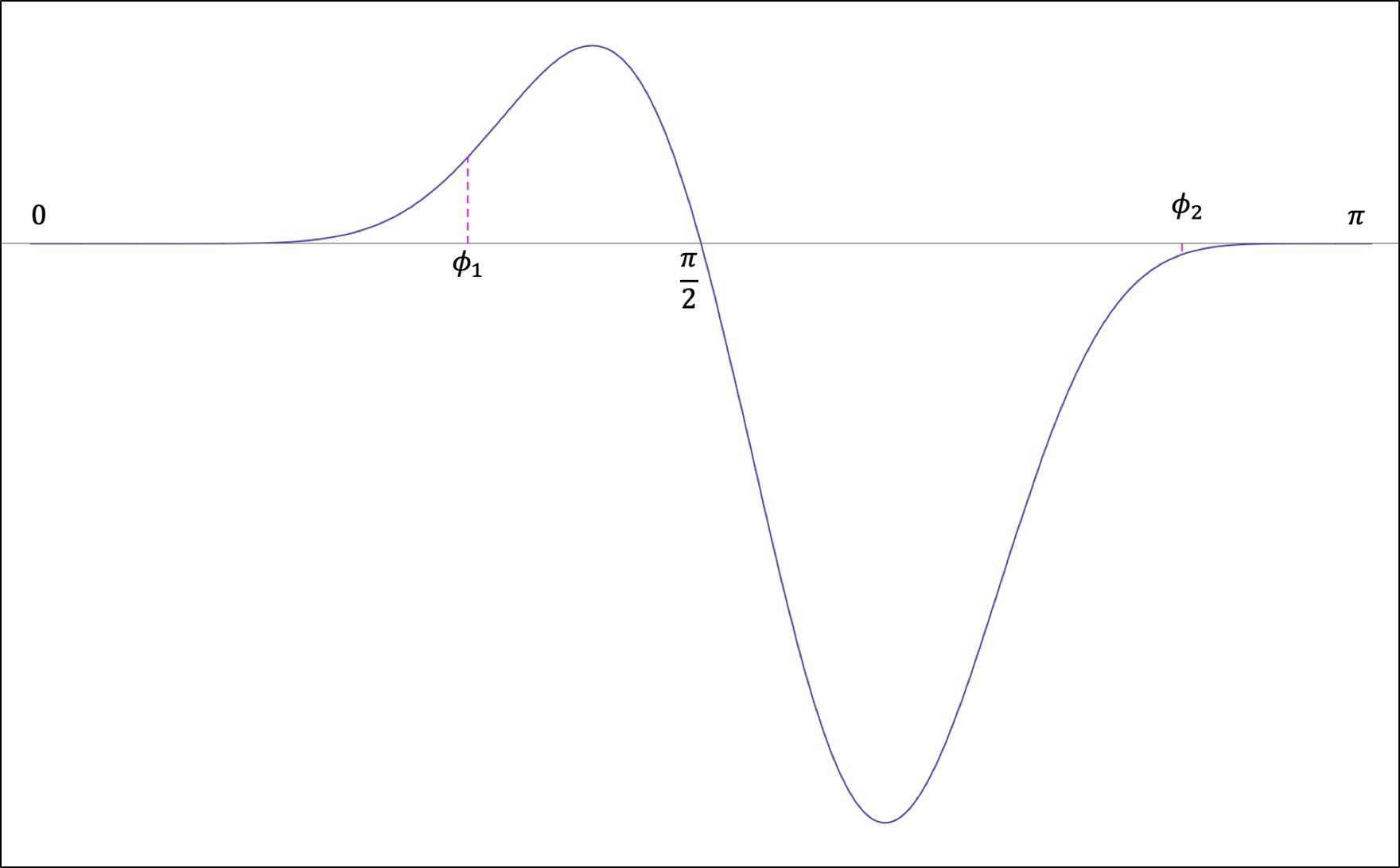}
	\caption{Graph of $g_2$ with $d=10:g_2(\phi_1)+g_2(\phi_2)>0$}
\end{figure*}

Let $\phi_1 \in (0,\pi/2)$ and $\phi_2 =\pi-\epsilon$ for some small $\epsilon$. Suppose that $f(\varphi)$ is given by the following formula, cf.~Figure~\ref{expl:smearysphere}:
\[f(\varphi)=
\begin{cases}
	c_1 \text{ if }\varphi \in [0,\phi_1],\\
	c_2 \text{ if } \varphi \in [\pi/2,\phi_2],\\
	0 \text{ otherwise.}
\end{cases}
\]
Then Eqs.~\eqref{expl5.6.1} and \eqref{expl5.6.2} are equivalent to
\begin{align} \label{condt1}
	\begin{split}
		c_1g_1(\phi_1 )+c_2g_1(\phi_2)-c_2\frac{\pi}{2} &=0,\\
		g_2(\phi_1)+g_2(\phi_2) &>0.
	\end{split}
\end{align}
\end{expl}
\begin{figure*}[h!]
\centering
	\includegraphics[width=0.5\textwidth]{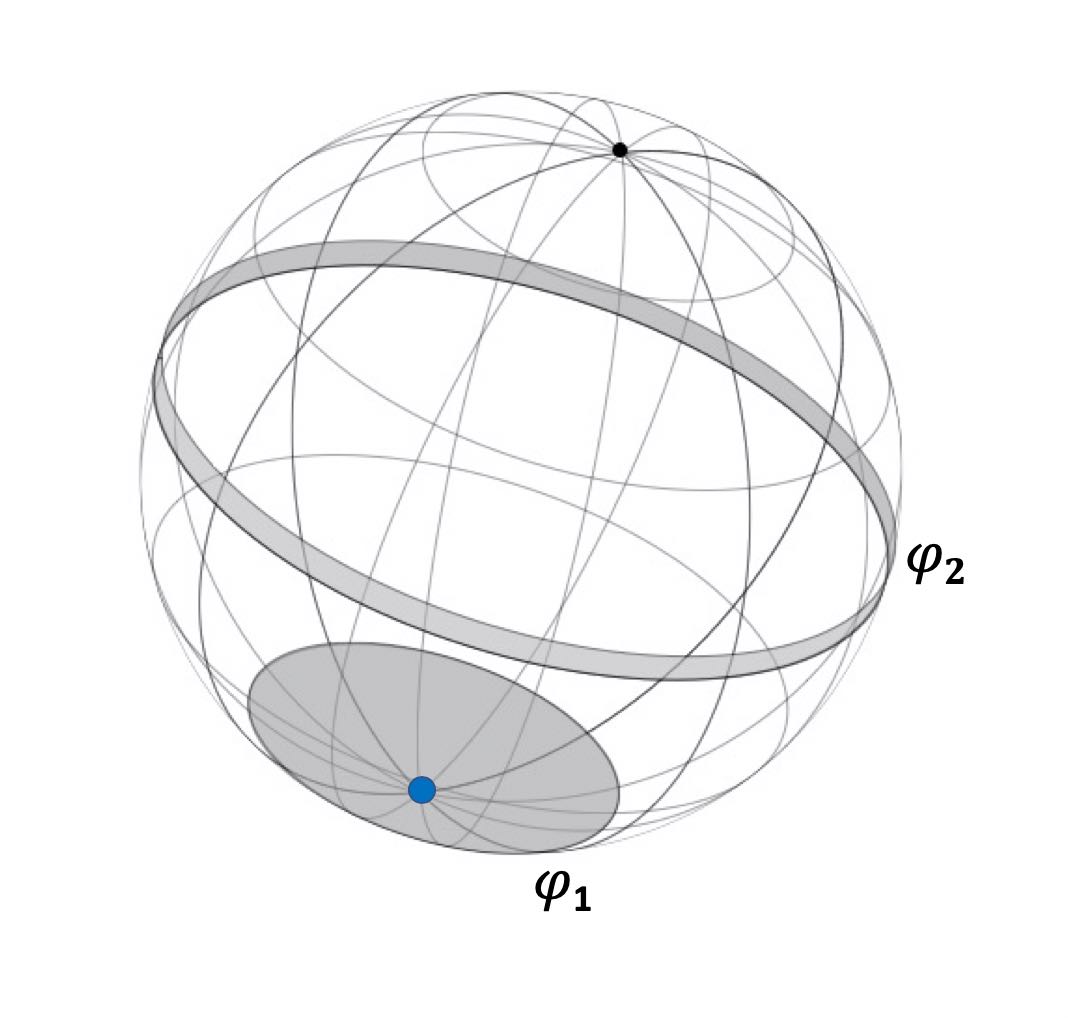}
	\vglue -0.2in
	\caption{An example of smeary CLT on spheres: the upper strip gets thinner as the dimension grows.}
	 \label{expl:smearysphere}
\end{figure*}

The second condition can be obtained by choosing $\epsilon$ sufficiently small. For example, select choose $\epsilon$ such that
\[ \bigg(\frac{\sin \phi_1}{\sin \epsilon}\bigg)^{d-3}=\frac{\pi}{\cos \phi_1(\sin \phi_1-\phi_1 \cos \phi_1)},\]
or 
\begin{equation} \label{epsiloncondt}
	\epsilon =\arcsin \Big(\sin \phi_1 \sqrt[d-3]{\cos \phi_1(\sin \phi_1-\phi_1 \cos \phi_1)} \Big).
\end{equation} 
The first condition in Eq.~\eqref{condt1} is equivalent to
\begin{equation} \label{condt2}
	\frac{c_1}{c_2}=\frac{\pi-2\epsilon\sin^{d-1}\epsilon}{2\phi_1 \sin^{d-1}\phi_1}. 
\end{equation}

To illustrate, the measure $\mu$ is uniformly distributed on a lower cup
from the south pole to the longitude of $\phi_1$ with a constant
distribution function $f_1=c_1$ and is uniformly distributed on a
upper strip from the equator to the longitude of $\pi-\epsilon$ with a
constant distribution function $f_2=c_2$.

\begin{remark} \label{r:smeary} 
 It follows from Eq.~\eqref{epsiloncondt} that if we fix $\phi_1$ then $ \epsilon \to \phi_1 $ and so $\phi_2 \to \pi-\phi_1$ $\text{ as } d \to \infty$.
	It means that in high dimension, smeary effect can happen with a
	uniformly distributed cup in the bottom and a uniformly distributed
	thin upper strip near the equator, cf.~Figure~\ref{expl:smearysphere}. For example, the case $\phi_1=0$ is studied in \cite{Elt19}. In the paper, Eltzner shows that $p$ is a global Fr\'echet mean for some $\phi_2$ and $\phi_2$ converges to $\pi/2$ as the dimension $d$ grows.
\end{remark}

%%%%%%%%%%%%%%%%%%%%%%%%%%%%%%%%%%%%%%%%%%%%%%%%%%%%%%%%%%%%%%%%%%%%%%%%%%%%%%%%%%%%%%%%%%%%%%%
%%%%%%%%%%%%%%%%%%%%%%%%%%%%%%%%%%%%%%%%%%%%%%%%%%%%%%%%%%%%%%%%%%%%%%%%%%%%%%%%%%%%%%%%%%%%%%%

%%%%%%%%%%%%%%%%%%%%%%%%%%%%%%%%%%%%%%%%%%%%%%%%%%%%%%%%%%%%%%%%%%%%%%%%%%%%%%%%%%%%%%%%%%%%%%
\section{Derivative tensors of the squared distance function} \label{s:derv_tensors}

In this section, derivative tensors of the squared distance function are computed formally. The Hessian of the squared distance function has been computed by applying Jacobi fields; see, for example, \cite{Afs09, BB08}. However, earlier computations such as \cite{Afs09, BB08} used a parametric approach which is difficult to apply in computing higher derivatives of the Fr\'echet function. Here we present a non-parametric approach. While the technique is also standard in differential geometry, its result allows us to compute higher derivatives of the Fr\'echet function at the mean.

\subsection{The squared distance function and Jacobi fields}
Let $\mathcal{X}(M)$ be the space of vector fields on $M$ and $\mathcal{D}(M)$ be the space of real valued functions on $M$ that are of class $C^{\infty}$ on $M \setminus \CC_y$. Let $Y$ denote the vector field $\nabla \rho_y$ on $M \setminus \CC_y$. Then for any point $x \in M \setminus \CC_y$ the vector $(-Y(x))\in T_xM$ is the tangent vector of the geodesic $\exp_x(-tY(x))$ connecting $x$ and $y$. 

Recall that the Hessian $\nabla^2 \rho_y$ at a point $x$ is a linear map $\nabla^2 \rho_y(x):T_xM \to T_xM$, which we write as $\nabla^2 \rho_y$ when $x$ is clear from the text, defined by the identity
		\[ \nabla^2\rho_y \cdot v =\nabla_v (\nabla \rho_y) =\nabla_v Y  \]
	for $v \in T_xM$. If $Z$ is a vector field in $M$ with $Z(x)=v$ and $\nabla_Z$ stands for the covariant derivative in direction $Z$ then
		\[ \nabla^2 \rho_y \cdot v = \nabla_v Y =\nabla_ZY(x). \]
	We wish to compute $\nabla_ZY$ at $x$
for any vector field $Z \in \mathcal{X}(M)$. If no confusion can arise, we write $\<Y,Z \>$ for $\mathfrak{g}(Y,Z)$ and $\norm{Y}$ for $\norm{Y(x)}$. Define
		\begin{equation} \label{e:g_k}
			\begin{split}
				g_{y}: M &\to \RR\\
					x &\mapsto 	\dd(x,y)\cot (\dd(x,y)) 
			\end{split}
		\end{equation}
The following result and its proof is familiar; see for example \cite[Theorem 2.4.1]{Afs09}, \cite[Theorem 2.2]{BB08}.
 \begin{proposition} \label{prop:2ndderv}
 The Hessian of $\rho_y$ is given by 
 	\begin{align} \label{2ndderv}
 		\begin{split}
 			\nabla^2\rho_y(Z)=\nabla_{Z}Y=Zg_{y}-Y\frac{\< Y,Z\>}{\norm{Y}^2}\Big(g_{y}-1 \Big),
 		\end{split}
 	\end{align}  
note that we omit the variable $x$ in the expressions of $Z(x)$, $g_{y}(x)$, and $Y(x)$ in the above formula.  
\end{proposition}

%%%%%%%%%%%%%%%%%%%%%%%%%%%%%%%%%%%%%%%%%%%%%%%%%%%%%%%

\subsection{Higher order derivatives of the squared distance fucntion}\label{ss:higher_derv}
In this section, we compute derivatives of order $3$ and $4$ of $\rho_y$ in Proposition~\ref{prop:3rdderv} and Proposition~\ref{prop:4thderv}. Following notations of derivative tensors in \cite[Sec. 4.5]{Car92}, we view $\nabla^2\rho_y$ as a $2$-tensor, which allows us to formally compute the derivative tensor of $\rho_y$ up to order $4$. 
\subsubsection{The third derivative tensor}
\begin{proposition}\label{prop:3rdderv}The third derivative tensor $\nabla^3\rho_y$ of $\rho_y$ is 
\begin{align} 
\begin{split} \label{thirdderv}
	\nabla^3\rho_y(W,Z,T)&= W\big(\nabla^2\rho_y(Z,T) \big)-\nabla^2\rho_y(\nabla_{W}Z,T)-\nabla^2\rho_y(Z,\nabla_{W}T)  \\
	&=\big( \< Z,T\>\< Y,W\>+\<W,T\>\< Y,Z\>+\< W,Z\>\< Y,T\>  \big)\frac{g_{y}-g_{y}^2}{\norm{Y}^2}\\
	&\hspace{1em}+\<Y,W\>\<Y,Z\>\<Y,T\>\frac{3g_{y}^2-3g_{y}+\norm{Y}^2}{\norm{Y}^4}-\<Z,T\>\<Y,W\>.
\end{split}
\end{align}
\end{proposition}
\begin{proof}

It follows from Eq.~\eqref{2ndderv} that
\begin{align*} 
	\nabla^2 \rho_y(Z,T)&=\nabla  Y(Z,T)\\
	&=Z\<Y,T \> -\< Y,\nabla_ZT\> \\
	&= \<\nabla_ZY,T \>\\
	&= \< Z,T\> g_{y}-\<Y,T \>\frac{\<Y,Z \>}{\norm{Y}^2}\Big(g_{y}-1 \Big),
\end{align*}
for any vector fields $Z, T \in \XX(S^d)$. Here recall that \[g_{y}=\dd(x,y)\cot(\dd(x,y))=\norm{Y} \cot (\norm{Y}).\] The gradient of $g_{y}$ can be formally computed as  
\begin{equation} \label{devg1}
	\nabla g_{y}=\frac{Y}{\norm{Y}}g_{y}',
\end{equation}
where $g_{y}'$ is the formal derivative
\[ g_{y}'=\cot(\norm{Y})-\frac{\norm{Y}}{\sin^2 (\norm{Y})}, \]
 of $g_{y}.$ Using the identity
\[ t(t\cot t)'=t \cot t-(t\cot t)^2-t^2 \]
we have
\begin{equation*}
	\norm{Y}g_{y}'=g_{y}-g_{y}^2-\norm{Y}^2. 
\end{equation*} 
Thus, Eq.~\eqref{devg1} can be rewritten as
\begin{equation} \label{devg}
	\nabla g_{y}=\frac{Y}{\norm{Y}^2}\Big( g_{y}-g_{y}^2-\norm{Y}^2 \Big).
\end{equation}
Using Eqs.~\eqref{2ndderv} and \eqref{devg}, we derive the formula \eqref{thirdderv} of the third derivative $\nabla^3\rho_y$ of $\rho_y$.
\end{proof}

\subsubsection{The fourth derivative tensor}
Set
\begin{align*}
	\begin{split}
		\boldsymbol I_1&= \< Z,T\>\< U,W\>+ \<W,Z \>\< U,T\>+\< W,T\>\< U,Z\>,\\
		\boldsymbol I_2&= \<Y,U \>\<Y,W\>\< Z,T\>+\<Y,U \>\<Y,Z\>\< W,T\>+\<Y,U \>\<Y,T\>\< Z,W\>,\\
		\boldsymbol I_3&= \<U,W \>\< Y,Z\>\< Y,T\>+\<U,Z \>\< Y,W\>\< Y,T\>+\<U,T \>\< Y,Z\>\< Y,W\>,\\
		\boldsymbol I_4&= \<Y,U\>\<Y,W \>\<Y,Z \>\< Y,T\>.
	\end{split}
\end{align*}
Then we have the following result about the fourth derivative $\nabla^4F_{\muu}$.
\begin{proposition}\label{prop:4thderv} 
The fourth derivative tensor $\nabla^4\rho_y$ of $\rho_y$ is 
\begin{align} 
	\begin{split} \label{fourthdev2}
		\nabla^4\rho_y(U,W,Z,T)&= \BI_1\frac{g_{y}^2-g_{y}^3}{\norm{Y}^2} +\frac{\BI_2}{\norm{Y}^2} \bigg( \frac{3g_{y}^3-3g_{y}^2}{\norm{Y}^2}-1 +2g_{y}\bigg)\\
		&\hspace{1em}+\frac{\BI_3}{\norm{Y}^2}\bigg( \frac{3g_{y}^3-3g_{y}^2} {\norm{Y}^2} +g_{y} \bigg)\\
		&\hspace{1em}+\frac{\BI_4}{\norm{Y}^4}\bigg( \frac{15g_{y}^2-15g_{y}^3}{\norm{Y}^2}+4-9g_{y} \bigg)\\
		&\hspace{1em}-\<Z,T \>\< U,W\>g_{y}+\frac{\< Z,T\>\< Y,W\>\<Y,U\>}{\norm{Y}^2}(g_{y}-1).
	\end{split} 
\end{align}
\end{proposition}
\begin{proof}
The fourth derivative tensor is the following $4$-tensor 
\begin{align*} 
	\begin{split}
		\nabla^4\rho_y(U,W,Z,T)&=U\big( \nabla^3\rho_y(W,Z,T)\big)-\nabla^3\rho_y(\nabla_UW,Z,T)-\nabla^3\rho_y(W,\nabla_UZ,T)\\
		&\hspace{1em}-\nabla^3\rho_y(W,Z,\nabla_UT).\\
	\end{split}
\end{align*}
To simplify notations, let us introduce
\begin{align*}
	\indent \boldsymbol\Sigma_1&=\< Z,T\>\< Y,W\>+\<W,T\>\< Y,Z\>+\< W,Z\>\< Y,T\>,\\
	\boldsymbol\Sigma'_1&= \< Z,T\>\< \nabla_UY,W\>+\<W,T\>\<  \nabla_UY,Z\>+\< W,Z\>\<  \nabla_UY,T\>,\\
	\boldsymbol\Sigma_2&= \<Y,W\>\<Y,Z\>\<Y,T\>,\\
	\boldsymbol\Sigma'_2&=\<\nabla_UY,W\>\<Y,Z\>\<Y,T\>+\<Y,W\>\<\nabla_UY,Z\>\<Y,T\>+\<Y,W\>\<Y,Z\>\<\nabla_UY,T\>.
\end{align*}
Then
\begin{align} 
	\begin{split} \label{fourthdrev1}
		\nabla^4\rho_y(U,W,Z,T)&=\boldsymbol\Sigma'_1 \frac{g_{y}-g_{y}^2}{\norm{Y}^2}+\boldsymbol\Sigma_1U\Big(\frac{g_{y}-g_{y}^2}{\norm{Y}^2} \Big)\\
			&\hspace{1em}+\boldsymbol\Sigma'_2\frac{3g_{y}^2-3g_{y}+\norm{Y}^2}{\norm{Y}^4}  + \boldsymbol\Sigma_2U\Big( \frac{3g_{y}^2-3g_{y}+\norm{Y}^2}{\norm{Y}^4}\Big) \\
			&\hspace{1em}-\<Z,T \>\<\nabla_UY,W \>.
	\end{split}
\end{align}
Equations~\eqref{2ndderv} and \eqref{devg} give us explicit expressions of $\boldsymbol \Sigma'_1$ and $\boldsymbol \Sigma'_2$:
\begin{align*}
	\begin{split}
		\boldsymbol\Sigma'_1=&\big( \< Z,T\>\< U,W\>+ \<W,Z \>\< U,T\>+\< W,T\>\< U,Z\> \big)g_{y}\\
		&+\frac{1-g_{y}}{\norm{Y}^2}\Big(\<Y,U \>\<Y,W\>\< Z,T\>+\<Y,U \>\<Y,Z\>\< W,T\>+\<Y,U \>\<Y,T\>\< Z,W\> \Big),\\
		\boldsymbol\Sigma'_2=&\big( \<U,W \>\< Y,Z\>\< Y,T\>+\<U,Z \>\< Y,W\>\< Y,T\>+\<U,T \>\< Y,Z\>\< Y,W\> \big)g_{y}\\
		&+3\<Y,U\>\<Y,W \>\<Y,Z \>\< Y,T\>\frac{1-g_{y}}{\norm{Y}^2}.\\
	\end{split}	
\end{align*}
In addition,
\begin{align} \label{Uf}
	\noindent U\Big( \frac{g_{y}-g_{y}^2}{\norm{Y}^2} \Big)=&\frac{\<Y,U \>}{\norm{Y}^4}\Big(2g_{y}^3-g_{y}^2-g_{y}-\norm{Y}^2+2\norm{Y}^2g_{y} \Big) .
\end{align}
Then $\boldsymbol{\Sigma'_1}$ and $\boldsymbol{\Sigma'_1}$ can be rewritten as 
\begin{align}
	\begin{split} \label{dervS}
		\boldsymbol\Sigma_1'&=g_{y} \boldsymbol I_1+(1-g_{y})\frac{\boldsymbol I_2}{\norm{Y}^2}, \\
		\boldsymbol\Sigma_2'&=g_{y}\boldsymbol I_3+3(1-g_{y})\frac{\boldsymbol I_4}{\norm{Y}^2}.
	\end{split}
\end{align}	
Eq.~\eqref{fourthdev2} now follows from Eqs.~\eqref{fourthdrev1}, \eqref{Uf}, and \eqref{dervS}.
\end{proof}
%%%%%%%%%%%%%%%%%%%%%%%%%%%%%%%%%%%%%%%%%%%%%%%%%%%%%%%
%Now we examine conditions which make the Fr\'echet function differentiable.
\section{Behaviors of Fr\'echet means and CLT on spheres}
\subsection{Differentiability conditions of the Fr\'echet function}\label{s:differentiability}
The main result in this section is to prove Proposition~\ref{p:diffcondprob}. Fix a point $p \in S^d$. Results from the previous section, namely Eqs.~\eqref{2ndderv}, \eqref{thirdderv}, and \eqref{fourthdev2}, imply that the squared distance function $\rho_y$ has the fourth derivative everywhere except at the cut locus $\CC_y$ of $y$. 
\begin{proof}[Proof of Proposition~\ref{p:diffcondprob}]
Suppose that $U,\ W,\ Z,\ T$ are unit vector fields in $\XX(S^d)$ and $p$ is a point in $S^d$. It follows from the Leibniz integral rule that the Fr\'echet function is differentiable of order $4$ at $p$ if the integrals:
%$\nabla^2\rho_y(p)$, $\nabla^3\rho_y(p)$ and $\nabla^4\rho_y(p)$ are integrable (in $y$), i.e. we want 
\begin{align} \label{tensor}
	\begin{split}
		\nabla F(T)_p&=\int_{S^d}\nabla \rho_y(T)_{p}d\mu(y);\\ 
		\nabla^2F(Z,T)_p&=\int_{S^d}\nabla^2\rho_y(Z,T)_pd\mu(y);\\
		\nabla^3F(W,Z,T)_p&=\int_{S^d} \nabla^3\rho_y(W,Z,T)_p d\mu(y); \text{ and }\\
		\nabla^4F(U,W,Z,T)_p&= \int_{S^d}\nabla^4\rho_y(U,W,Z,T)_p d\mu(y)
	\end{split}
\end{align}
converge.
Choose a polar parametrization on $S^d$ so that $p$ has coordinate $(0,0,\ldots, 0)$. Let $\Omega:=[0,\pi]^{d-1}\times [0,2\pi]$ and define the following parametrization of $S^d$
\begin{align} \label{polarpara}
	\begin{split}
		\pp : \Omega &\to S^{d} \subset \RR^{d+1} \\
		\boldsymbol \varphi=(\varphi_1,\ldots,\varphi_{d-1},\varphi_{d}) &\mapsto (x_1,\ldots,x_{d+1}), \\
		x_1&=\cos \varphi_1, \\
		x_2&=\sin \varphi_1 \cos \varphi_2, \\
		&\ \vdots\\
		x_{d+1}&= \sin \varphi_1 \sin \varphi_2 \ldots \sin \varphi_{d},
	\end{split}
\end{align}
where $\varphi_1 \ldots \varphi_{d-1} \in [0,\pi], \ \varphi_{d}\in [0,2\pi].$ So $p$ has coordinate $(0,0, \ldots ,0)$ and its antipode $\CC_p$ has coordinate $(\pi, 0,\ldots ,0)$. If a point $y$ has coordinate $(\varphi_1,\ldots,\varphi_d)$ then
\begin{equation}
	\dd_y(p)=\norm{Y(p)}=\varphi_1 \label{simy}
\end{equation} 
and
\begin{equation}
	 g_{y}(p)=\norm{Y(p)}\cot(\norm{Y(p)})=\varphi_1\cot \varphi_1. \label{simg}
\end{equation}
Under the parametrization $\pp$, the volume element of $S^d$ is 
\[ dV=\sin ^{d-1}\varphi_1 \ldots \sin \varphi_{d-1}d \boldsymbol \varphi.\]
Let $f_p$ be the pullback of $f$ under the parametrization $\pp$ in Eq.~\eqref{polarpara}. Then
	\begin{equation} \label{eq:f_p}
		 d\mu=f_p(\boldsymbol \varphi)\sin ^{d-1}\varphi_1 \ldots\sin \varphi_{d}d \boldsymbol \varphi. 
	\end{equation}
It follows from Eqs.~\eqref{2ndderv}, \eqref{thirdderv} and \eqref{fourthdev2} that as $y$ approaches the antipode of $p$, the term $g_{y}^{j-1}(p)$ is the only unbounded one in the expression of $\nabla^j\rho_y(p)d\mu(y)$. Hence, the integrability of $\nabla^j\rho_y(p)d\mu(y)$
depends on the integrability of $g_{y}^{j-1}(p)$ for
$j=1,\ldots,4$. It then suffices to require that the integral
\begin{align} 
	\begin{split} \label{diffcond}
		\int_{S^d}g^{j-1}_{y}(p)d\mu(y)
		&=\int_{\Omega} \varphi^{j-1}_1\cot^{j-1}(\varphi_1)f_p(\boldsymbol\varphi)\sin ^{d-1}\varphi_1 \ldots \sin \varphi_{d}d\boldsymbol\varphi\\
		&=\int_{\Omega} \varphi^{j-1}_1\cos^{j-1}(\varphi_1)f_p(\boldsymbol\varphi)\sin ^{d-j}\varphi_1 \sin^{d-2}\varphi_2 \ldots \sin \varphi_{d}d \boldsymbol\varphi
	\end{split}
\end{align} 
 converges for $j=1,2 \ldots, 4$.
 
Let $\Omega_{\epsilon}=\{ \varphi \in \Omega :\varphi_1 \in (\pi-\epsilon,\pi ] \} $ be the preimage under the parametrization map $\pp$ of $B_{\epsilon}(\CC_p)$. Conditions \hyperref[cond:i]{(\emph{i})} and \hyperref[cond:ii]{(\emph{ii})} translate into 
		\begin{itemize} \label{condition3.2}
			\item[$(i')$.] $f_p$ is bounded on $\Omega_{\epsilon}$ when $d\geq j$,
			\label{cond:i1}
			\item[$(ii')$.] $f_p$ vanishes on $\Omega_{\epsilon}$ when $d<j$. \label{cond:ii1}
		\end{itemize} 
		First, assume that $d\geq j$ and $f_p$ is bounded on $\Omega_{\epsilon}$, which is condition \hyperref[cond:i1]{$(i')$}, and write the right integral in Eq.~\eqref{diffcond} as 
		\begin{align*}
			&\int_{\Omega} \varphi^{j-1}_1\cos^{j-1}(\varphi_1)f_p(\varphi)\sin ^{d-j}(\varphi_1)\sin^{d-2}\varphi_2 \ldots \sin \varphi_{d}d\boldsymbol \varphi \\
			&\hspace{.5em}=\int_{\Omega_{\epsilon}} \varphi^{j-1}_1\cos^{j-1}(\varphi_1)f_p(\varphi)\sin ^{d-j}(\varphi_1) \sin^{d-2}\varphi_2 \ldots \sin \varphi_{d}d \boldsymbol\varphi\\
			&\hspace{2.5em}+\int_{\Omega_{\epsilon}^C} \frac{\varphi^{j-1}_1}{\sin^{j-1}(\varphi_1)}\cos^{j-1}(\varphi_1)f_p(\varphi)\sin ^{d-1}\varphi_1 \sin^{d-2}\varphi_2 \ldots \sin \varphi_{d}d \boldsymbol\varphi\\
			&\hspace{.5em}=I_1+I_2.
		\end{align*}
		The integral $I_2$ converges as $\frac{\varphi^{j-1}_1}{\sin^{j-1}(\varphi_1)}\cos^{j-1}(\varphi_1)$ is bounded on $\Omega_{\epsilon}^C$. For the convergence of $I_1$, note that $f_p$ is bounded on $\Omega_{\epsilon}$, hence the function under the integral of $I_1$ is bbounded on $\Omega_{\epsilon}$. Thus the integral in Eq.~\eqref{diffcond} converges.
		
		Now suppose that condition \hyperref[cond:ii1]{$(ii')$} holds. That means $d<j$ and $f_p(\boldsymbol\varphi)=0$ for all $\boldsymbol\varphi \in \Omega_{\epsilon}$. Then the integral in Eq.~\eqref{diffcond} reduces to just $I_2$ and hence converges.
		        
		We have proved that the Fr\'echet function $F$ is differentiable of order $j$ at $p$. Next, observe that conditions \hyperref[cond:i]{(\emph{i})} and \hyperref[cond:ii]{(\emph{ii})} still hold if we replace $p$ by any point $q \in B_{\epsilon}(p)$. The arguments above then can be applied to $q$, so $F$ is differentiable of order $j$ at $q$. Thus $F$ is differentiable of order $j$ on $B_{\epsilon}(p)$.
\end{proof}

%%%%%%%%%%%%%%%%%%%%%%%%%%%%%%%%%%%%%%%%%%%%%%%%%%%%%%%

\subsection{Behavior of local Fr\'echet means}\label{s:behavior}

We give a proof of Proposition~\ref{prop:derivative_tensors} in this section. Note that the \fc mean is assumed to be unique and Assumption~\ref{assumption:1} is in effect---that is, $\mu$ is absolutely continuous and rotationally symmetric about a point $p$.

Since $\mu$ is rotationally symmetric about $p$, so is the Fr\'echet function $F(x)$. Combining with the assumption that $F(x)$ is differentiable in a neighborhood of $p$, it suffices to study the behavior of $F(x)$ along a direction starting from $p$.

\begin{proof}[Proof of Proposition~\ref{prop:derivative_tensors}]
The first order tensor vanishes due to symmetry. We proceed to compute higher order tensors.
Recall from the previous section that $Y$ is the vector field $\nabla \rho_y$. Given a vector field $Z$, let $\theta(x)=\angle(Y(x),Z(x))$ be the angle between $Y$ and $Z$. It follows from Eqs.~\eqref{2ndderv}, \eqref{thirdderv} and \eqref{fourthdev2} that
\begin{align} \label{e:deriv_tensors_sphere}
	\begin{split}
		\nabla^2\rho_y(Z,Z)&=\norm{Z}^2\big(g_{y}\sin^2 \theta +\cos^2 \theta \big),\\
		\nabla^3\rho_y(Z,Z,Z)&=\frac{\norm{Z}^3}{\norm{Y}}\big(3g_{y}-3g_{y}^2-\norm{Y}^2 \big)\big( \cos \theta -\cos^3 \theta \big),\\
		\nabla^4\rho_y(Z,Z,Z,Z)&= \frac{\norm{Z}^4}{\norm{Y}^2}\Bigg( 3g_{y}^2-3g_{y}^3-\norm{Y}^2 g_{y}\\
		&\hspace{4em}+\cos^2\theta \Big(18g_{y}^3-18g_{y}^2+10\norm{Y}^2g_{y}-4\norm{Y}^2 \Big)\\
		&\hspace{4em}+\cos^4 \theta \Big( 15g_{y}^2-15g_{y}^3-9\norm{Y}^2g_{y}+4\norm{Y}^2 \Big)  \Bigg)\\
		&=\frac{\norm{Z}^4}{\norm{Y}^2}\Bigg(-\sin^2\theta \Big( 12g_{y}^2-12g_{y}^3 -8\norm{Y}^2g_{y}+4\norm{Y}^2\Big)  \\
		&\hspace{4em}+ \sin^4 \theta \Big(15g_{y}^2-15g_{y}^3-9\norm{Y}^2g_{y}+4 \norm{Y}^2 \Big) \Bigg).
	\end{split}
\end{align}Assume, without loss of generality, that $\theta =\varphi_2$ and $\norm{Z(p)}=1$. Recall from Eq.~\eqref{simy} and Eq.~\eqref{simg} that
\begin{align*}
	\norm{Y(p)}&=\varphi_1,\\
	g_{y}(p)&=\varphi_1\cot \varphi_1.
\end{align*}
Thus the second tensor of $F(x)$ is rewritten as follows.
\begin{align*}
	\begin{split}
		\nabla^2F(Z,Z)_p&=\int_{\Omega}\Big(\varphi_1 \cot\varphi_1 \sin^2 \varphi_2+\cos^2 \varphi_2\Big)f_p(\varphi_1) \sin^{d-1}\varphi_1  \ldots \sin \varphi_n d\boldsymbol\varphi\\
		&=V(S^{d-2})\int_0^{\pi}\int_0^{\pi} \Big(\varphi_1\cot \varphi_1 \sin^{d}\varphi_2\\
		&\text{\hspace{3cm}}+\cos^2\varphi_2\sin^{d-2}\varphi_2  \Big)\sin^{d-1}\varphi_1f_p(\varphi_1)d\varphi_1 d\varphi_2,
	\end{split}
\end{align*}
where $V(S^k)$ is the volume of the $k$-dimensional unit sphere. 
Set $a_k=\int_0^{\pi}\sin^k xdx$ and use the following identity
\[  \sin ^kx-\frac{k-1}{k}\sin^{k-2}x=-\frac{d}{dx} \Big( \frac{1}{k}\cos x \sin^{k-2}x\Big), \]
we have that 
\[ a_k=\frac{k-1}{k}a_{k-2}.\]
Write $\nabla^2F(Z,Z)_p$ as 
\begin{align*}
	\begin{split}
		\nabla^2F(Z,Z)_p&=V(S^{d-2})\int_0^{\pi}\Big(a_d \varphi_1\cot \varphi_1+a_{d-2}-a_d \Big)\sin^{d-1}\varphi_1f(\varphi_1)d\varphi_1\\
		&=V(S^{d-2})\int_0^{\pi}\frac{a_d}{d-1}\Big((d-1)\varphi_1\cos \varphi_1 \sin^{d-2} \varphi_1 +\sin^{d-1}\varphi_1\Big)f(\varphi_1)d\varphi_1\\
		&=\frac{V(S^{d-2})a_{d-2}}{d} \int_0^{\pi}f(\varphi_1)d(\varphi_1 \sin^{d-1}\varphi_1).\\
	\end{split}
\end{align*}
The identity $V(S^d)=a_{d-1}V(S^{d-1})$ yields 
\begin{equation}\label{2nddervre}
	\nabla^2F(Z,Z)_p=\frac{V(S^{d-1})}{d}\int_0^{\pi}f(\varphi_1)d(\varphi_1 \sin^{d-1}\varphi_1).
\end{equation}
It is not hard to verify that the third order tensor vanishes due to symmetry.
\begin{align} \label{3rddervfin}
	\begin{split}
		\nabla^3F(Z,Z,Z)_p&=V(S^{d-2})\int_{0}^{\pi}\int_0^{\pi}G(\varphi_1) \cos{\varphi_2}\sin^{d}\varphi_2d\varphi_2d\varphi_1 \\
		&=0,\\
	\end{split}
\end{align}
with $G(\varphi_1)$ is some function in $\varphi_1$.

\noindent With a little more computation we get the result for the fourth order derivative tensor.
\begin{equation} \label{fourthdervre}
	\begin{split}
			\nabla^4F(Z,Z,Z,Z)_p&=\frac{a_dV(S^{d-2})}{d+2}\int_0^{\pi}\Big(\big( \sin \varphi\cos \varphi -\varphi \cos^3\varphi\big)\big(3d-9 \big)\\&\hspace{7em}+\varphi \cos \varphi \sin^2 \varphi \big( 7-d \big) -4\sin^3\varphi \Big)\sin^{d-4}\varphi f(\varphi)d\varphi.
	\end{split}
\end{equation}
Observe that for $d=2,3$ and $\varphi \in [0,\pi],$
\[ \big( \sin \varphi\cos \varphi -\varphi \cos^3\varphi\big)\big(3d-9 \big)+\varphi \cos \varphi \sin^2 \varphi \big( 7-d \big) -4\sin^3\varphi <0. \] 
Therefore for $d < 4,$
\begin{equation} \label{nlessthan4}
	\nabla^4F(Z,Z,Z,Z)_p<0.
\end{equation}
For $d \geq 4$ the tensor in Eq.~\eqref{fourthdervre} becomes
\begin{align} \label{fourthderv3}
	\begin{split}
		\nabla^4F(Z,Z,Z,Z)_p=\frac{a_dV(S^{d-2})}{d+2}\int_0^{\pi}f(\varphi)d\Big(\sin^{d-3}\varphi \big(2\varphi\sin^2\varphi+3\cos \varphi \sin \varphi-3\varphi \big) \Big).
	\end{split}
\end{align}

Next, due to symmetry, we show that the derivative tensors in Proposition~\ref{prop:derivative_tensors} have no other components. We present a proof only for the second order tensor; the arguments for higher order tensors are the same.

Suppose that two vector fields $Z,T \in \XX(S^d)$ are orthonormal at $p$. Then Eq.~\eqref{2ndderv} gives
\begin{align*} 
	\begin{split}
		\nabla^2\rho_y(Z,T)=-\frac{\<Y,Z\>\<Y,T\>}{\norm{Y}^2}(g_y-1),\\
%		\nabla^3F(Z,Z,T)&=0,\\
%		\nabla^4F(Z,Z,Z,T)&=0.
	\end{split}
\end{align*}
and hence
\begin{align}\label{6.1.1}
	\nabla^2F(Z,T)_p=-\int_{S^d}\frac{\<\exp^{-1}_p(y),Z \>_p \<\exp^{-1}_p(y),T\>_p}{\dd^2(y,p)}(g_y(p)-1)d\mu(y).
\end{align}
Let $A$ be a reflection on $T_pS^d$ that fixes $Z(p)$ and sends $T(p)$ to $-T(p)$. Since $Q$ is rotationally symmetric about $p$, it is fixed under the exponential of $A$ denoted by 
	\begin{align*}
		\tilde{A}=\exp (A): S^d &\to S^d\\
					y&\mapsto\exp_p(A\exp^{-1}_p(y)).
	\end{align*}
In other words, $d\mu(y)=d\mu\big(\tilde{A}(y) \big)$. For the sake of simplicity, write $\tilde{A}y$ for $\tilde{A}(y)$. Since $\dd(y,p)=\dd(\tilde{A}y,p)$,

\begin{align*}
	\nabla^2F(Z,T)_p&=-\int_{S^d}\frac{\<\exp^{-1}_p(y),Z \>_p \<\exp^{-1}_p(y),T\>_p}{\dd^2(y,p)}(g_y(p)-1)d\mu(y)\\
	&=-\int_{S^d}\frac{\<\exp^{-1}_p(\tilde{A}y),Z \>_p \<\exp^{-1}_p(\tilde{A}y),T\>_p}{\dd^2(\tilde{A}y,p)}(g_{\tilde{A}y}(p)-1)d\mu(\tilde{A}y)\\
	&=-\int_{S^d}\frac{\<A\exp^{-1}_p(y),Z \>_p \<A\exp^{-1}_p(y),T\>_p}{\dd^2(y,p)}(g_{y}(p)-1)d\mu(y)\\
	&=-\int_{S^d}\frac{\<\exp^{-1}_p(y),AZ \>_p \<\exp^{-1}_p(y),AT\>_p}{\dd^2(y,p)}(g_{y}(p)-1)d\mu(y)\\
	&=\int_{S^d}\frac{\<\exp^{-1}_p(y),Z \>_p \<\exp^{-1}_p(y),T\>_p}{\dd^2(y,p)}(g_{y}(p)-1)d\mu(y)\\
	&=-\nabla^2F(Z,T)_p.
\end{align*}
Thus $\nabla^2F(Z,T)_p=0$. Similarly, $\nabla^3F(Z,Z,T)_p=0$ and $\nabla^4(Z,Z,Z,T)=0.$ These results, in conjunction with Eqs.~\eqref{2nddervre},~\eqref{3rddervfin},~\eqref{fourthdervre}, and \eqref{nlessthan4} yields the results in Proposition~\ref{prop:derivative_tensors}.
\end{proof}

\begin{remark}
	This method of
applying Jacobi fields with a non-parametric approach to derive derivatives of squared distance
function could be applied to symmetric spaces since Jacobi fields in
such spaces can be explicitly expressed. However, behaviors of
Frech\'et mean on symmetric spaces still remains unknown because
derivatives of the squared distance function do not translate well to
corresponding derivatives of Frech\'et function as it is in the case
of spheres.

\end{remark}
\section*{Acknowledgements}\label{s:acknowledgements}
A great debt goes to Stephan Huckemann and Benjamin Elztner for enlightening conversations on smeariness on spheres. 
The author gratefully thanks his advisor, Ezra Miller, for invaluable advice and comments thoughout the project.
The author was supported by NSF DMS-1702395 for conference travel.

%%%%%%%%%%%%%%%%%%%%%%%%%%%%%%%%%%%%%%%%%%%%%%%%%%%%%%%%%%%%%%%%%%%%%%%%%%%%%%%%%%%%%%%%%%%%%%%
%%%%%%%%%%%%%%%%%%%%%%%%%%%%%%%%%%%%%%%%%%%%%%%%%%%%%%%%%%%%%%%%%%%%%%%%%%%%%%%%%%%%%%%%%%%%%%%

%%%%%%%%%%%%%%%%%%%%%%%%%%%%%%%%%%%%%%%%%%%%%%%%%%%%%%%%%%%%%%%%%%%%%%%%%%%%%%%%%%%%%%%%%%%%%%%

%%%%%%%%%%%%%%%%%%%%%%%%%%%%%%%%%%%%%%%%%%%%%%%%%%%%%%%

%%%%%%%%%%%%%%%%%%%%%%%%%%%%%%%%%%%%%%%%%%%%%%%%%%%%%%%%%%%%%%%%%%%%%%%%%%%%%%%%%%%%%%%%%%%%%%%
%%%%%%%%%%%%%%%%%%%%%%%%%%%%%%%%%%%%%%%%%%%%%%%%%%%%%%%%%%%%%%%%%%%%%%%%%%%%%%%%%%%%%%%%%%%%%%%
%%%%%%%%%%%%%%%%%%%%%%%%%%%%%%%%%%%%%%%%%%%%%%%%%%%%%%%%%%%%%%%%%%%%%%%%%%%%%%%%%%%%%%%%%%%%%%%
%%%%%%%%%%%%%%%%%%%%%%%%%%%%%%%%%%%%%%%%%%%%%%%%%%%%%%%%%%%%%%%%%%%%%%%%%%%%%%%%%%%%%%%%%%%%%%%
%%%%%%%%%%%%%%%%%%%%%%%%%%%%%%%%%%%%%%%%%%%%%%%%%%%%%%%%%%%%%%%%%%%%%%%%%%%%%%%%%%%%%%%%%%%%%%%
%%%%%%%%%%%%%%%%%%%%%%%%%%%%%%%%%%%%%%%%%%%%%%%%%%%%%%%%%%%%%%%%%%%%%%%%%%%%%%%%%%%%%%%%%%%%%%%
%%%%%%%%%%%%%%%%%%%%%%%%%%%%%%%%%%%%%%%%%%%%%%%%%%%%%%%%%%%%%%%%%%%%%%%%%%%%%%%%%%%%%%%%%%%%%%%
\pagebreak

\end{document}